\documentclass{article}
\usepackage{amsmath, amsfonts, amsthm, amssymb}
\usepackage{enumerate}

\theoremstyle{plain}
\newtheorem{proposition}{Proposition}[section]
\newtheorem{theorem}[proposition]{Theorem} 
\newtheorem{lemma}[proposition]{Lemma}
\newtheorem{conjecture}[proposition]{Conjecture} 
\newtheorem{corollary}[proposition]{Corollary}

\theoremstyle{definition}

\newtheorem{remark}[proposition]{Remark}
\newtheorem{ex}[proposition]{Example}
\newtheorem{conv}[proposition]{Convention}

\begin{document}

\title{On the generalized Davenport constant and the Noether number}   
\author{K\'alm\'an Cziszter $^a$ 
\thanks{The paper is based on results from the PhD thesis of the first author written at the Central European University.}
\qquad and \qquad M\'aty\'as Domokos $^b$ 
\thanks{The second author is partially supported by OTKA  NK81203 and K101515.}
}
\date{}
\maketitle 
{\small \begin{center} 
$^a$ Central European University, Department of Mathematics and its Applications, 
N\'ador u. 9, 1051 Budapest, Hungary \\
Email: cziszter\_kalman-sandor@ceu-budapest.edu \\
 $^b$ R\'enyi Institute of Mathematics, Hungarian Academy of Sciences,\\
 Re\'altanoda u. 13-15, 1053 Budapest, Hungary \\
Email: domokos.matyas@renyi.mta.hu
\end{center}
}

\maketitle

\begin{abstract} Known results on the generalized Davenport constant relating zero-sum sequences over a finite abelian group are extended for the generalized Noether number relating rings of polynomial invariants of an arbitrary finite group. 
An improved general upper degree bound for polynomial invariants of a non-cyclic finite group that cut out the zero vector is given.  
\noindent 2010 MSC: 13A50 (Primary) 11B75, 13A02 (Secondary)

\noindent Keywords: Noether number, Davenport constant, polynomial invariants
\end{abstract}

\section{Introduction}  

The \emph{Davenport constant} ${\mathsf{D}}(A)$ of a finite abelian group $A$ is defined as the smallest positive integer $n$ such that every sequence over $A$ of length at least $n$ has a non-empty zero-sum subsequence.  The Davenport constant naturally appears in the theory of polynomial invariants of finite groups. 
The \emph{Noether number} $\beta(G)$ of a finite group $G$ is the maximal possible degree of  an indecomposable polynomial invariant of $G$ (the definition involves a fixed base field ${\mathbb{F}}$ which we always assume to have characteristic not dividing the order of $G$). It was shown by B. Schmid  \cite{schmid} that when $G=A$ is abelian, then 
$\beta(G)={\mathsf{D}}(A)$. 

Halter-Koch \cite{halter-koch} introduced for any positive integer $k$ the 
\emph{generalized Davenport constant} ${\mathsf{D}}_k(A)$ as the smallest positive integer $n$ such that 
every sequence over $A$ of length at least $n$ is divisible by the product of $k$ non-empty zero-sum subsequences (cf. Chapter 6.1 in \cite{geroldinger-halterkoch}; 
by the \emph{product} of sequences over $A$ we mean their concatenation, and \emph{divisibility} of sequences is defined accordingly). 
This notion also can be seen as  the abelian special case of a concept of invariant theory. 
In \cite{cziszter-domokos:1} the authors introduced the \emph{generalized Noether number} $\beta_k(G)$ for an arbitrary finite group $G$ and  positive integer $k$ (the definition again involves a fixed base field ${\mathbb{F}}$ of characteristic not dividing  $|G|$, suppressed from the notation).  
When $G=A$ is abelian, then $\beta_k(A)={\mathsf{D}}_k(A)$. As it is demonstrated in \cite{cziszter-domokos:1}, the main use of the generalized Noether number is that it can be efficiently used to derive upper bounds for the ordinary Noether number of a group in terms of the generalized Noether numbers of its subquotients (see the beginning of Section~\ref{sec:growth} for some details).  

In the first part of the present  note we develop further the analogy between the generalized Davenport constant and the generalized Noether number. 
Theorem~\ref{thm:beta GxH} gives a lower bound for the generalized Noether number of a direct product of groups in terms of the generalized Noether numbers of the factors, 
generalizing thereby  Lemma 6.1.4 from \cite{geroldinger-halterkoch} for non-abelian groups. 
Next we investigate the behaviour of $\beta_k(G)$ as a function of $k$. 
The fact that $D_k(A)$ is an almost linear function of $k$ was shown by M. Freeze and W. A. Schmid \cite{freeze-schmid}, building on a result of Delorme, Ordaz and Quiroz \cite{delorme}; see also Theorem 6.1.5 in \cite{geroldinger-halterkoch}. 
This is generalized in  Corollary~\ref{cor:growthbetag}, which we derive  from basic generalities on graded rings and the method of polarization in invariant theory. It states  that for a fixed finite group $G$ and a base field ${\mathbb{F}}$ of characteristic zero there exists a positive integer $k_0$ and a non-negative integer $\beta_0(G)$ such that for all $k\ge k_0$ we have $\beta_k(G)=k\sigma(G)+\beta_0(G)$;  the number $\sigma(G)$ is another well known quantity of invariant theory: it is the minimal positive integer $n$ such that for any representation of $G$ there exist homogeneous $G$-invariant polynomial functions of degree at most $n$ such that their only common zero is the origin. 
(In the special case when $G=A$ is abelian, we have $\sigma(A)=\mathrm{exp}(A)$, the exponent of $A$. We note also that for linearly reductive algebraic groups a related quantity $\sigma(G,V)$ plays a significant role in constructive invariant theory, see \cite{popov} or \cite{derksen}.) 

In Section~\ref{sec:sigmabasic} we  establish some basic properties of $\sigma(G)$ for $G$ finite. They are used to prove Theorem~\ref{thm:mainsigma}, stating  that 
if $G$ is non-cylic, then $\sigma(G)\le |G|/q$ for the minimal prime divisor of $|G|$, provided that ${\mathrm{char}}({\mathbb{F}})$ does not divide $|G|$. 
This is 
an easier but  stronger variant for $\sigma(G)$ of the main combined result on $\beta(G)$ proved in 
\cite{cziszter-domokos:1} and \cite{cziszter-domokos:2} (continuing the investigations of \cite{schmid}, \cite{domokos-hegedus}, \cite{sezer}). 


\section{Preliminaries}\label{sec:prel} 

We need to recall some generalities on graded modules. 
By a graded module here we mean an $\mathbb{N}$-graded module $M=\bigoplus_{d=0}^\infty M_d$ over a 
commutative graded ${\mathbb{F}}$-algebra $R=\bigoplus_{d=0}^\infty R_d$ such that $R_0= {\mathbb{F}}$ is a field when $R$ is unital 
and $R_0 = \{0\}$ otherwise; in the latter case still we assume that $M$ is an  ${\mathbb{F}}$-vector space, the multiplication map is ${\mathbb{F}}$-bilinear, and a submodule or an ideal by definition is assumed to be a subspace. 
We set 
$M_{\le s} := \bigoplus_{d=0}^s M_d$, and  
$R_+:=\bigoplus_{d>0}R_d$. We write $R_+^l$ for the $l$th power of the ideal $R_+$, and more generally, for subsets $A,B\subset R$, 
$AB$ stands for the $F$-subspace spanned by $\{ab\mid a\in A, b\in B\}$, and $\langle A\rangle$ denotes the ideal in $R$ generated by $A$. 
The subalgebra of $R$ generated by $R_{\le s}:=\bigoplus_{d=0}^sR_d$ will be denoted by ${\mathbb{F}}[R_{\le s}]$. 

The factor space $M/R_+M$ inherits the grading. 
Define 
\[\beta(M,R):=\mbox{ the top degree of }M/R_+M\] 
provided that $M/R_+M$ has finitely many non-zero homogeneous components, and write $\beta(M,R) = \infty$ otherwise. 
By the graded Nakayama Lemma $\beta(M,R)$ is the minimal non-negative integer $s$ such that 
$M=M_{\le s}+M_{\le s}R$ (i.e. $M$ is generated as an $R$-module by $M_{\le s}$), when $M$ is generated in bounded degree. 

In particular, the maximal degree of a homogeneous element of $R_+$
not belonging to $R_+^2$ is 
\[\beta(R) := \beta(R_+, R),\]  
the minimal positive integer $n$ such that $R$ is generated as an ${\mathbb{F}}$-algebra by homogeneous elements of degree at most $n$. 

Let $M$ be a graded $R$-module.  We define
for any integer $k \ge 1$
\[ \beta_k(M, R) : = \beta(M, R_+^k)\]
The abbreviation 
\[\beta_k(R):= \beta_k(R_+,R)\] 
will also be used. 
In the special case $k=1$ we recover $\beta_1(M,R)=\beta(M,R)$ and $\beta_1(R)=\beta(R)$. 
Note the trivial inequality $\beta_k(R)\le k\beta(R)$. 

The graded modules and algebras we are interested in come from invariant theory. 
We fix a base field ${\mathbb{F}}$, and let $G$ be a finite group whose order is invertible in  ${\mathbb{F}}$. 
Take a $G$-module $V$, i.e. a finite dimensional ${\mathbb{F}}$-vector space endowed with a representation of $G$ on $V$. 
The coordinate ring ${\mathbb{F}}[V]$ is defined 
in abstract terms as the symmetric tensor algebra of the dual space $V^*$. 
So ${\mathbb{F}}[V]$ is  a polynomial ring in $\dim(V)$ variables,
hence in particular it is a graded ring with  ${\mathbb{F}}[V]_1 = V^*$. 
The left action of $G$ on $V$ induces  a natural right action on $V^*$ given by  
$x^g(v) = x(gv)$ for any $g\in G, v\in V$ and $x\in V^*$. 
This right action of $G$ on $V^*$  extends to the symmetric tensor algebra ${\mathbb{F}}[V]$. 
The corresponding  \emph{ring of polynomial invariants} is 
\[{\mathbb{F}}[V]^G := \{ f \in {\mathbb{F}}[V]: f^g = f \quad \forall g \in G \}\]
and 
\[\beta(G,V) := \beta({\mathbb{F}}[V]^G)\] 
is called the \emph{Noether number} of the $G$-module $V$. 
We also set  
\[\beta_k(G,V):= \beta_k({\mathbb{F}}[V]^G)\] 
for an arbitrary positive integer $k$,   
and 
\[\beta_k(G) := \sup\{\beta_k(G,V)\mid V\text{ is a finite dimensional }G\text{-module over }{\mathbb{F}}\}.\] 
We shall refer to these numbers as the {\it generalized Noether numbers} of the group $G$. 

The famous theorem of E. Noether asserts that $\beta(G):=\beta_1(G)$ is bounded by the order of $G$. 
When ${\mathrm{char}}({\mathbb{F}})=0$ or ${\mathrm{char}}({\mathbb{F}})>|G|$, this was proved in \cite{noether:1916}. 
The result was extended to non-modular positive characteristic independently by Fleischmann \cite{fleischmann} and Fogarty \cite{fogarty}.

\begin{conv}\label{conv}
Throughout this paper  ${\mathbb{F}}$ is our base field and
 $G$ (or $H$) is a finite group of order not divisible by  ${\mathrm{char}} ({\mathbb{F}})$, unless explicitly stated otherwise. 
By a  $G$-module we mean a finite dimensional ${\mathbb{F}}$-vector space endowed with a linear action of $G$.  
\end{conv}


\section{Lower bound for direct products}

\begin{lemma}\label{lemma:beta-beta*} 
For any $G$-module $V$ there exists an irreducible $G$-module $U$ such that 
\[\beta_k (G,V\oplus U)\geq \beta_k({\mathbb{F}}[V],{\mathbb{F}}[V]^G)+1.\]
\end{lemma} 

\begin{proof} Write $L:= {\mathbb{F}}[V]$, $ R:= {\mathbb{F}}[V]^G$ and set $d:=\beta_k(L,R)$. 
By complete reducibility of the $G$-module $L_d$ its submodule $R_+^kL \cap L_d$ has a  direct complement,
which is non-zero by the definition of $d$, hence it contains an irreducible $G$-submodule $U$. 
Choose a basis $e_1,\dots,e_n$ in $U$ and let $\varepsilon_1,\dots,\varepsilon_n$ be the corresponding dual basis in $U^*$. 
The matrix of $g$ acting on $U^*$ via the contragredient representation is the transpose of the inverse of the matrix of $g$ acting on $U$. 
Thus $f:=\sum_{i=1}^ne_i\varepsilon_i$, viewed as an element  in the polynomial ring ${\mathbb{F}}[V\oplus U]= {\mathbb{F}}[V] \otimes {\mathbb{F}}[\varepsilon_1,\dots,\varepsilon_n]$, 
is a $G$-invariant  of degree $d+1$.     
We claim that $f \not\in S_+^{k+1}$ where $S := {\mathbb{F}}[V \oplus U]^G$. 
Note that the action of $G$ on ${\mathbb{F}}[V\oplus U]$ preserves the   total degree both in the variables belonging to $V^*$ and to $U^*$. 
Suppose indirectly that $f\in S_+^{k+1}$. 
Then $f=\sum_j g_jh_j$ where $g_j \in R_+^k$ while $h_j\in S_+$ is linear on $U$, 
i.e. $h_j =\sum_{i=1}^n h_{j, i}\varepsilon_i$ for some polynomials $h_{j,i}\in L$. 
After equating the coefficients of $\varepsilon_i$ on both sides we get that 
$e_i = \sum_j g_jh_{j,i} \in R_+^kL$, contradicting the choice of $U$. 
\end{proof} 

 \begin{corollary}\label{cor:beta-beta*}
If $V$ is a $G$-module such that $\beta_k(G,V)=\beta_k(G)$ then 
\[\beta_k(G,V)=\beta_k({\mathbb{F}}[V],{\mathbb{F}}[V]^G)+1.\]
\end{corollary}

\begin{proof}
For any $G$-module $V$ it holds that $\beta_k(G,V)\leq \beta_k(L,R)+1$ where $L ={\mathbb{F}}[V]$ and $R = {\mathbb{F}}[V]^G$. 
Indeed, if $f \in L$ has degree $\deg(f) > \beta_k(L,R)+1$ then  $f\in L_+\cdot(\bigoplus_{d>\beta_k(L,R)}L_d) \subseteq L_+R_+^k$. 
Apply the transfer map $\tau:L \to R$, $f\mapsto \sum_{g\in G}f^g$. It is a graded $R$-module epimorphism from $L$ onto $R$ (see for example Chapter 1 in \cite{benson}).  It follows that $\tau(f) \in R_+^{k+1}$, and on the other hand $\tau(L)=R$, hence the desired inequality follows. 
The reverse inequality  is an immediate consequence of Lemma~\ref{lemma:beta-beta*}. 
\end{proof}

\begin{remark} Compare Corollary~\ref{cor:beta-beta*} with the formula ${\mathsf{D}}_k(A)=\mathsf{d}_k(A)+1$ in Lemma 6.1.2 of  
\cite{geroldinger-halterkoch}, where $\mathsf{d}_k(A)$ is the maximal length of a sequence over $A$ which is not divisible by the product of $k$ zero-sum subsequences. 
It is not difficult to show that $\mathsf{d}_k(A)=\sup_V\beta_k({\mathbb{F}}[V],{\mathbb{F}}[V]^A)$ where $V$ ranges over all $A$-modules, whereas 
$\beta_k(A)={\mathsf{D}}_k(A)$, as we mentioned in the introduction. 
\end{remark}

\begin{theorem}\label{thm:beta GxH}
For any integers $r,s \ge 1$ and finite groups $G,H$  we have the inequality 
\[\beta_{r+s-1}(G\times H) \geq \beta_r(G) + \beta_s(H)-1.\] 
\end{theorem} 

\begin{proof} 
If $M$ and $N$ are graded modules over the graded algebras $R$ and $S$, respectively, then:
\begin{align} \label{beta GxH}
\beta_{r+s-1}(M \otimes N, R \otimes S) \ge \beta_r(M, R) + \beta_s(N, S)
\end{align}
Indeed, there are elements $x\in M_{\beta_r(M,R)}\setminus R_+^rM$ and $y\in N_{\beta_s(N,S)}\setminus S_+^sN$. 
Take an ${\mathbb{F}}$-vector space basis $\mathcal{B}_1$  of $R_+^rM$, and extend $\mathcal{B}_1\cup\{x\}$ to a basis $\mathcal{B}$ of $M$. 
Similarly, let $\mathcal{C}_1$ be a basis of $S_+^sN$, and extend $\mathcal{C}_1\cup\{y\}$ to a basis $\mathcal{C}$ in $N$. 
Then  $\mathcal{A}:=\{u\otimes v\mid u\in\mathcal{B}_1,v\in\mathcal{C}\text{ or }u\in\mathcal{B},v\in\mathcal{C}_1\}$ is a basis of 
$T:= R_+^rM\otimes N+M\otimes S_+^sN$. 
On the other hand $\mathcal{A} \cup \{x\otimes y\}$ is part of the basis $\{u\otimes v\mid u\in \mathcal{B},v\in\mathcal{C}\}$ of $M\otimes N$, 
showing that $x \otimes y \not\in T$. But $ T \supseteq (R\otimes S)_+^{r+s-1}(M\otimes N)$ and $\deg(x\otimes y)=\beta_r(M,R)+\beta_s(N,S)$, whence \eqref{beta GxH} readily follows.

Now take a $G$-module $V$ with $\beta_r(G,V)=\beta_r(G)$, 
and an  $H$-module $W$ with $\beta_s(H,W)=\beta_s(H)$. 
Given that ${\mathbb{F}}[V \oplus W]^{G \times H} = {\mathbb{F}}[V]^G \otimes {\mathbb{F}}[W]^H$ we have the following sequence of inequalities:
\begin{align*}
\beta_{r+s-1}(G \times H) - 1 
&\ge   \beta_{r+s-1} ({\mathbb{F}}[V\oplus W],{\mathbb{F}}[V\oplus W]^{G\times H}) 	&& \text{by Lemma~\ref{lemma:beta-beta*} }\\
&\ge \beta_r({\mathbb{F}}[V],{\mathbb{F}}[V]^G) + \beta_s({\mathbb{F}}[W],{\mathbb{F}}[W]^H) && \text{by \eqref{beta GxH} } \\ 
&= \beta_{r}(G) + \beta_s(H) -2 && \text{by Corollary~\ref{cor:beta-beta*} } 
\end{align*}

\end{proof}

\begin{remark}
In the special case $G=A$ abelian we recover  Lemma 6.1.4 from \cite{geroldinger-halterkoch}. 
\end{remark}


\section{The growth rate of $\beta_k(G)$}\label{sec:growth}

In the study of the Noether bound for finite groups, the following inequalities due to Schmid \cite{schmid} in characteristic zero and extended to positive non-modular characteristic by Sezer \cite{sezer}, Fleischmann \cite{fleischmann:2}, Knop \cite{knop} are very useful: 
For a normal subgroup  $N$ in $G$ and an arbitrary subgroup $H$ in $G$ we have $\beta(G)\le [G:H]\beta(H)$ and $\beta(G)\le \beta(N)\beta(G/N)$. 
Our motivation to study $\beta_k(G)$ stems from the following strengthening proved in \cite{cziszter-domokos:1}: 
\begin{align*}\beta_k(G)&\le \beta_{\beta_k(G/N)}(N) \\
\beta_k(G)&\le \beta_{k[G:H]}(H)\end{align*} 
The  estimates for $\beta(G)$ obtained in terms of generalized Noether numbers of its subquotient $K$  using the latter inequalities  are better than the estimates derived from the original inequalitites, as soon as $\beta_k(K)$ is strictly smaller than $k\beta(K)$. 
A partial theoretical explanation of the experience that indeed, typically $\beta_k(G)$ is strictly smaller than $k\beta(G)$ and its extent is obtained in this section. 

We start by studying  in general for a fixed commutative graded ${\mathbb{F}}$-algebra $R$
the behavior of $\beta_k(R)$ as a function of $k$. 
The surjection $R_+/R_+^{k+1} \to R_+/R_+^k$ shows that $\beta_{k}(R) \le \beta_{k+1}(R)$ for all $k$. 
We note that  $\beta_k(R)$ is not always a strictly increasing function of $k$: 

\begin{ex}
Consider the ring $R = {\mathbb{F}}[a,b] / (b^3-a^9, ab^2-a^7)$
and define a grading by setting  $\deg(a)=1$ and $\deg(b)=3$. 
Then $b^2\in R_+^2\setminus R_+^3$, and $b^2$ spans the degree $6$ homogeneous component of $R_+^2/R_+^4$.  
In this case for all $l\geq 7$ we have that $R_l\subseteq R_+^5$, hence 
$6=\beta_2(R)=\beta_3(R)=\beta_4(R)$. 
\end{ex}

On the other hand, for a fixed $G$-module $V$, $\beta_k(G,V)$  is unbounded because of the following trivial observation: 
\begin{lemma} \label{beta unbounded}
$\beta_k(R)$  as a function of $k$ is bounded  if and only if  
there is an integer $n$ such that $R_i = \{ 0 \}$ for all $i \ge n$. 
\end{lemma}
\begin{proof}
 Note that $R_+^{n+1}\subseteq \bigoplus_{d\ge n+1}R_d$. Hence if
$R_n\neq \{0\}$, then  $R_n\nsubseteq R_+^{n+1}$, implying  $\beta_n(R)\ge n$.
Conversely,  if  $R_i = \{ 0 \}$ for all $i \ge n$ then $\beta_i(R)< n$.
\end{proof}

\begin{lemma}\label{lemma:k/r} 
For any positive integers $r\leq k$ we have the inequality 
\[\beta_k(R)\leq \frac kr \beta_r(R).\]  
\end{lemma} 

\begin{proof} Suppose to the contrary that $\beta_k(R)>\frac kr \beta_r(R)$. 
Then there exist homogeneous elements  $f_1,\ldots,f_l\in R_+$  
such that $l\le k$, 
$f:=f_1\cdots f_l$ is not contained in $R_+^{l+1}$, and $\deg(f)>\frac kr \beta_r(R)$ 
(this forces that $l> r$). We may suppose that $\deg(f_1)\geq\cdots\geq \deg(f_l)$. 
Then 
\[\frac{\deg(f_1)+\dots+\deg(f_r)}{r}\ge \deg(f_r)\ge \deg(f_{r+1})\ge\frac{\deg(f_{r+1})+\dots+\deg(f_l)}{l-r}\] 
hence 
\[\deg(f_1\cdots f_r)\ge \frac{r}{l}\deg(f_1\dots f_l)\ge \frac{r}{k}\deg(f)>\beta_r(R).\] 
It follows 
that $h:=f_1\cdots f_r\in R_+^{r+1}$, hence 
$f=hf_{r+1}\cdots f_l\in R_+^{l+1}$, a contradiction. 
\end{proof} 

By Lemma~\ref{lemma:k/r} the sequence $\frac{\beta_k(R)}{k}$ is monotonically decreasing,
and as it is also non-negative, it must converge to a certain limit. Our next goal  will be to clarify what is the value
of this limit.
For  a graded finitely generated commutative $\mathbb{F}$-algebra $R$ with $R_0=\mathbb{F}$ set 
\begin{align*}
\sigma(R) & :=\min\{d \in \mathbb{N}: R \text{ is finitely generated as a module over } {\mathbb{F}}[R_{\le d}]  \} 
\end{align*}  
Equivalently, $\sigma(R)$ is the minimal integer $d$ such that $\beta(R_+, {\mathbb{F}}[R_{\le d}])$ is finite.

\begin{proposition} \label{prop:sigma also}
Let $R$ be a finitely generated commutative graded ${\mathbb{F}}$-algebra. For any positive integer $k$ we have $\beta_k(R) \ge k\sigma(R)$.
\end{proposition}
\begin{proof}
It is well known that given a set $h_1,\dots,h_s\in R$ of homogeneous elements,  $R$ is a finitely generated module over its subalgebra 
${\mathbb{F}}[h_1,\dots,h_s]$ if and only if $R_+= \sqrt{\langle h_1,\dots,h_s\rangle }$, the radical of the ideal generated by the $h_i$. 
If $R$ is finite dimensional then $\sigma(R)=0$ and our statement obviously holds. Suppose $\dim_{{\mathbb{F}}}(R)=\infty$, hence 
$\sigma(R)>0$. 
By definition of $t:=\sigma(R)$  we have $R_+\neq \sqrt{\langle \bigoplus_{d=1}^{t-1}R_d\rangle}$, on the other hand 
$R_+=\sqrt{\langle \bigoplus_{d=1}^tR_d\rangle}$, hence there exists  an $f\in R_t$ with $f\notin \sqrt{\langle \bigoplus_{d=1}^{t-1}R_d\rangle}$. 
We claim that $f^k\notin R_+^{k+1}$. Indeed, $f^k\in R_+^{k+1}$ would imply  $f^k\in \langle \bigoplus_{d=1}^{t-1}R_d\rangle$, contrary to the choice  of $f$. 
Thus $\beta_k(R)\ge \deg(f^k)=k\sigma(R)$. 
\end{proof}
By definition of $\sigma(R)$ the number 
 \[ \eta(R) := \beta(R_+,  {\mathbb{F}}[R_{\le \sigma(R)}])\] 
is finite. 
Moreover, any homogeneous element $f \in R$ with $\deg(f) > \eta(R)$ belongs to the ideal $(\bigoplus_{d=1}^{\sigma(R)} R_d)\cdot R_+$, hence
$\beta(R) \le \eta(R)$ and 
more generally, by induction on $k$ one obtains 
\[\beta_k(R)\le (k-1)\sigma(R)+\eta(R).\]
We know from Lemma~\ref{beta unbounded} that  an integer $k_0$ exists such that 
 $\beta_k(R)  \ge \eta(R) - \sigma(R) $ holds
for any $k \ge k_0$. Hence if  $\deg(f) > \beta_k(R)+\sigma(R)$ then $f \in R$ 
can be written in the form $\sum_i g_ih_i$ where $0<\deg(g_i) \le \sigma(R)$ and $\deg(h_i) > \beta_k(R)$,
whence $h_i \in R_+^{k+1}$ and $f \in R_+^{k+2}$. 
This  argument  shows that for any $k\ge k_0$ we have
\begin{align} \label{sigma primko} \beta_{k+1}(R) \le \beta_k(R) + \sigma(R) \end{align}
 This simple observation  immediately leads us to the following result: 

\begin{proposition}\label{prop:linearity} 
Let $R$ be a finitely generated commutative graded ${\mathbb{F}}$-algebra. 
Then there are non-negative integers $k_0(R)$ and $\beta_0(R)$ such that 
\[ \beta_k(R) = k\sigma(R) + \beta_0(R) \quad \text{ for every } k > k_0(R).\]
\end{proposition}

\begin{proof}
Consider the sequence of integers $a_k := \beta_k(R) - k\sigma(R)$, where $k=k_0,k_0+1,\dots$ and 
$\beta_{k_0}(R)\ge \eta(R)-\sigma(R)$. 
By \eqref{sigma primko} it is monotonically decreasing and by Proposition~\ref{prop:sigma also} it is non-negative, 
therefore it stabilizes after finitely many steps, and this is what has been claimed. 
\end{proof}

For any $G$-module $V$ we  write 
\[\sigma(G,V) := \sigma({\mathbb{F}}[V]^G).\]
This quantity was much studied for $G$ a linearly reductive group  (see e.g. \cite{derksen} or \cite{popov}) and has the following well-known interpretation 
by the Hilbert Nullstellensatz: 

\begin{proposition} \label{prop:hilbert zero}
$\sigma(G,V)$ is the minimal positive integer $n$ such that there exists a subset of ${\mathbb{F}}[V]^G_+$ consisting of homogeneous elements with degree at most $n$,  whose  common zero locus in 
$\bar{\mathbb{F}} \otimes _{{\mathbb{F}}} V$ is $\{0\}$ (where $\bar{\mathbb{F}}$ stands for the algebraic closure of ${\mathbb{F}}$).
\end{proposition} 

Supposing $|G|\in{\mathbb{F}}^\times$ we have 
$\sigma(G,V) \le \beta(G,V) \le |G|$, and define 
\[ \sigma(G) := \sup\{\sigma(G,V)\mid V\mbox{ is a }G\mbox{-module}\}.\]
(In fact the inequality $\sigma(G)\le |G|$ holds in the modular case ${\mathrm{char}}({\mathbb{F}})\mid |G|$ as well, see Remark~\ref{remark:modularsigma} (ii).) 
As an immediate corollary of Proposition~\ref{prop:linearity} we obtain that for any $G$-module $V$ there exist non-negative integers 
$k_0(G,V)$ and $\beta_0(G,V)$ such that for all $k\ge k_0(G,V)$ we have 
$\beta_k(G,V)=k\sigma(G,V)+\beta_0(G,V)$. 
In characteristic zero the following can be proved: 

\begin{corollary}\label{cor:growthbetag} 
Suppose ${\mathrm{char}}({\mathbb{F}})=0$ and let $G$ be a finite group. 
There exist non-negative integers $k_0(G)$ and $\beta_0(G)$ such that for all 
$k\ge k_0(G)$ we have 
\[\beta_k(G)=k\sigma(G)+\beta_0(G).\] 
In particular, 
\[\lim_{k\to\infty}\frac{\beta_k(G)}k=\sigma(G).\]
\end{corollary} 

\begin{proof} Denoting by $V_{\mathrm{reg}}$ the regular representation of $G$, we have that $\beta_k(G)=\beta_k(G,V_{\mathrm{reg}})$ holds for all $k$
by the same argument as in the proof of the special case $k=1$ in  \cite{schmid} based on Weyl's theorem on polarization (cf. \cite{weyl}). 
Hence the statement holds by Proposition~\ref{prop:linearity}. 
\end{proof}

\begin{remark}\label{remark:growthrate} 
As we mentioned in the Introduction, Corollary~\ref{cor:growthbetag} in the special case when $G=A$ is abelian is due to 
M. Freeze and W. A. Schmid \cite{freeze-schmid}, and Delorme,  Ordaz and Quiroz \cite{delorme}. For some results on $\eta(A)$ see e.g. \cite{geroldinger-halterkoch} ch. 5.7. 
 \end{remark} 


\section{Some basic properties of $\sigma(G)$}\label{sec:sigmabasic}

In this section we collect some basic statements about $\sigma(G)$ that we will need to prove Theorem~\ref{thm:mainsigma}.

\begin{lemma}\label{lemma:sigma irred}
Let $V_1,..., V_n$ be any $G$-modules and $W=V_1\oplus\dots\oplus V_n$. Then  
\[ \sigma(G, W) = \max_{i=1}^n \sigma(G, V_i) \] 
In particular $\sigma(G) = \max_U \sigma(G,U)$ where $U$ ranges over all isomorphism classes of irreducible $G$-modules.
\end{lemma}
\begin{proof} 
Let $R = {\mathbb{F}}[W]^G$ and denote  by $S_i$ the subalgebra of ${\mathbb{F}}[V_i]^G$ generated by its elements of degree at most $\sigma(G,V_i)$. 
As  ${\mathbb{F}}[W] =\otimes_{i=1}^n {\mathbb{F}}[V_i]$ is obviously finitely generated as a 
 $\otimes_{i=1}^nS_i$-module, and $\otimes_{i=1}^nS_i \subseteq {\mathbb{F}}[R_{\le d}]$ where $d := \max_{i=1}^n\sigma(G,V_i)$,
it follows that $\sigma(G,W)\le d$. 

For the reverse inequality let $T= {\mathbb{F}}[V_i]^G$ for a fixed $i$
and observe that the restriction to $V_i$ gives  a graded algebra surjection $\psi: R \to T$.  
Hence the image under $\psi$ of a finite set of module generators of $R$ over its subalgebra ${\mathbb{F}}[R_{\le\sigma(R)}] $ 
must  generate $T= \psi(R)$ as a module over its subalgebra $\psi({\mathbb{F}}[R_{\le\sigma(R)}])= {\mathbb{F}}[T_{\le\sigma(R)}]$, as well.
In particular $\sigma(G,V_i)\leq\sigma(G,W)$.
\end{proof}

\begin{remark} The number $\sigma({\mathbb{F}}[W]^G)$ when $G$ is a linearly reductive group acting algebraically on $W$ plays important role in 
finding explicit upper bounds for $\beta({\mathbb{F}}[W]^G)$, see \cite{popov} and \cite{derksen}. Lemma~\ref{lemma:sigma irred} is special for finite groups, and does not hold in general for reductive algebraic groups, when it may well happen that ${\mathbb{F}}[V]^G={\mathbb{F}}$, but ${\mathbb{F}}[V\oplus\dots\oplus V]^G$ contains non-constant elements 
(for example, take as $V$ the natural module ${\mathbb{F}}^n$ over $G=SL_n({\mathbb{F}})$). 
\end{remark} 

For an abelian group $A$ denote by $\exp(A)$ the least common multiple of the orders of the elements of $A$. 

\begin{corollary} \label{cor: sigma abelian}
Let $A$ be an abelian group and suppose that ${\mathbb{F}}$ is algebraically closed of characteristic not dividing $|A|$. Then 
\[ \sigma(A) = \exp(A).\]
\end{corollary}
\begin{proof}
Lemma~\ref{lemma:sigma irred} asserts that $\sigma(A)=\max_U \sigma(A,U)$ 
where $U$ runs through the  irreducible representations of $A$. These are all $1$-dimensional, 
and if $U^*= \langle x \rangle$ then ${\mathbb{F}}[x]^A = {\mathbb{F}}[x^e]$ where  $e \in \mathbb{N}$ is the order of the character ${\mathrm{char}}i:A\to {\mathbb{F}}^\times$ defined 
by $x^a={\mathrm{char}}i(a)x$ ($a\in A$). This readily implies our claim, as $A\cong \hat A$, where $\hat A:=\hom_{\mathbb{Z}}(A,{\mathbb{F}}^\times)$ is the group of characters of $A$. 
\end{proof}

\begin{lemma} \label{lemma:sigma red1}
Let  $N$ be a normal subgroup of $G$ and  $V$ a $G$-module.
Then 
\begin{align*}
\sigma(G,V) &\le \sigma(G/N)\sigma(N,V).  
\end{align*}
\end{lemma}
\begin{proof} Set  $W:=\bigoplus_{d=1}^{\sigma(N,V)}{\mathbb{F}}[V]^N_d$ 
and denote by $S$ the subalgebra of ${\mathbb{F}}[V]^N$ generated by $W$. 
Then $S$ is a finite module over its finitely generated subalgebra $S^{G/N}=S^G$, and ${\mathbb{F}}[V]^N$ is a finite $S$-module, thus 
${\mathbb{F}}[V]^N$ is a finite, hence noetherian $S^G$-module, implying in turn that its submodule ${\mathbb{F}}[V]^G$ is also a finite $S^G$-module. 
Write $\pi$ for the ${\mathbb{F}}$-algebra surjection ${\mathbb{F}}[W^*]\to S$ induced by the natural isomorphism between the linear component $(W^*)^*$ of the polynomial ring 
${\mathbb{F}}[W^*]$ and $W\subset{\mathbb{F}}[V]^N$. By linear reductivity of $G/N$, $\pi$ maps ${\mathbb{F}}[W^*]^{G/N}$ onto $S^G$. Let $T$ be the ${\mathbb{F}}$-subalgebra of ${\mathbb{F}}[W^*]$ generated by 
$\bigoplus_{d=1}^{\sigma(G/N,W^*)}{\mathbb{F}}[W^*]^{G/N}_d$. Then ${\mathbb{F}}[W^*]^{G/N}$ is a finite $T$-module, implying that $S^G$ is a finite $\pi(T)$-module, and thus ${\mathbb{F}}[V]^G$ is a finite $\pi(T)$-module. Since by construction $\pi(T)$ is generated by elements of degree at most $\sigma(N,V)\sigma(G/N,W^*)$, the desired inequality follows. 
\end{proof}

\begin{lemma} \label{lemma:sigma red2}
Let $G$ be a finite group,  $H$  a subgroup of $G$, and $V$ a $G$-module. 
Then 
\begin{align*}
\sigma(H,V)  \le \sigma(G,V)  \le [G:H] \sigma(H,V). 
\end{align*}
\end{lemma}

\begin{proof} The first inequality is trivial. By Proposition~\ref{prop:hilbert zero} there are homogeneous elements $f_1,\dots,f_r\in{\mathbb{F}}[V]^H$ of degree at most $\sigma(H,V)$ such that the common zero locus of $f_1,\dots,f_r$ in $\bar{\mathbb{F}}\otimes V$ is $\{0\}$. 
The formula $0=\prod_g(f_i-f_i^g)$ where $g$ ranges over a set of right $H$-coset representatives in $G$ shows that 
$f_i^{[G:H]}$ is contained in the ideal of ${\mathbb{F}}[V]$ generated by ${\mathbb{F}}[V]^G_+$. Since $\deg(f_i^{[G:H]})\le [G:H]\sigma(H,V)$, 
it follows that the common zero locus of $\bigoplus_{d=1}^{ [G:H]\sigma(H,V)}{\mathbb{F}}[V]^G_d$ is contained in the common zero locus $\{0\}$ of $\{f_i^{[G:H]}\mid i=1,\dots,r\}$. 
Consequently, again by Proposition~\ref{prop:hilbert zero}, ${\mathbb{F}}[V]^G$ is a finitely generated module over its subalgebra generated by 
the homogeneous components ${\mathbb{F}}[V]^G_d$ with $d\le [G:H]\sigma(H,V)$. 
\end{proof} 

\begin{remark}\label{remark:modularsigma} 
(i) The statement and proof of Lemma~\ref{lemma:sigma red1} remain valid under the weaker assumption that $[G:N]$ is not divisible by ${\mathrm{char}}({\mathbb{F}})$. 

(ii) The statement and proof of Lemma~\ref{lemma:sigma red2} remain valid in the modular case ${\mathrm{char}}({\mathbb{F}})\mid |G|$. 
When $H=\{1\}$ is the trivial subgroup, we obtain the inequality $\sigma(G,V)\le |G|$. 
\end{remark} 

Lemma~\ref{lemma:sigma red1} and Lemma~\ref{lemma:sigma red2} have the folowing immediate corollary: 
\begin{corollary} \label{cor:sigma subquotient}
For any subquotient $K$ of $G$ we have 
\[ \frac{\sigma(G)}{|G|} \le \frac{\sigma(K)}{|K|}\]
\end{corollary} 

\begin{remark} Lemma~\ref{lemma:sigma red1} and Lemma~\ref{lemma:sigma red2} are the analogues for $\sigma$ of the reduction lemmata for $\beta$ in  \cite{schmid} mentioned at the beginning of Section~\ref{sec:growth}. For a variant of these statements concerning separating invariants see Section 3.9.4 in \cite{derksen-kemper}, 
\cite{kemper}, and \cite{kohls-kraft}. 
\end{remark} 

\section{$\sigma(G)$ for some semidirect products }

We need some facts and terminology relating zero-sum sequences over $A=Z_p$, the group of prime order $p$, written additively. See for example \cite{geroldinger-gao} 
as a general reference to this topic. 
Recall that by a \emph{sequence over} $A$ we mean a  sequence $S = (s_1, ..., s_d)$ of elements $s_i\in A$ where the order of the elements is disregarded and repetition is allowed.  
We say that $S$ is a \emph{zero-sum sequence} if $\sum_{i=1}^ds_i=0\in A$.  Denote by ${\mathrm{supp}}(S)$ the set of elements of $A$ that occur in $S$.

\begin{lemma}\label{lemma:tarto} 
For any non-empty subset $S \subseteq Z_p\setminus\{0\}$ there exists a 
zero-sum sequence $T$  of length at most $p$ with ${\mathrm{supp}}(T) = S$. 
\end{lemma} 

\begin{proof} Let $s_1,\dots,s_k$ denote the elements of $S$. If $s_1+\dots+s_k=0$, then the sequence $T:=(s_1,\dots,s_k)$ satisfies the requirements. 
Otherwise for $i=1,\dots k$ denote $n_i$ the unique element in $\{1,\dots,p-1\}$ with $n_is_i=-(s_1+\dots+s_k)\in Z_p$. 
The $n_i$ are distinct, hence the smallest among them, say $n_1\le p-k$. 
Then the sequence $T:=(s_1,\dots,s_1,s_2,\dots,s_k)$ (where the multiplicity of $s_1$ is $n_1+1$) satisfies the requirements.  
\end{proof} 

\begin{proposition}\label{sigma ZpZq}
Let $G=Z_p\rtimes Z_d$ be a semidirect product of cyclic groups, where $p$ is a prime, $d$ is a divisor of $p-1$, and $Z_d$ acts faithfully (via conjugation) on $Z_p$. Then we have 
$\sigma(G)=p$.
\end{proposition}
\begin{proof}
We know that $\sigma(G)\ge\sigma(Z_p)=p$ by Lemma~\ref{lemma:sigma red2} and Corollary~\ref{cor: sigma abelian}.
By Lemma~\ref{lemma:sigma irred}  it is enough to prove that $\sigma(G, U) \le p$, where $U$ is an irreducible $G$-module. 
Since $\sigma(G)$ is not sensitive for extending the base field, we may assume that ${\mathbb{F}}$ is algebraically closed. 
 Denote by $A$ the maximal normal subgroup $Z_p$ of $G$. 
 Then $G$ has only two types of irreducible representations: 
if $U$ is $1$-dimensional with $A$ in its kernel, then $\sigma(G, U) \le |G/A|  = |Z_d| \le  p-1$. 
Otherwise $U$ is induced from a non-trivial $1$-dimensional $A$-module. In this case we may choose variables 
$x_1,\dots,x_d$ in ${\mathbb{F}}[U]$ such that the $x_i$ are $A$-eigenvectors permuted up to scalar multiples by $G$, and 
denoting by $\theta_i\in\hat A:=\hom_{\mathbb{Z}}(A,{\mathbb{F}}^\times)$ the corresponding character of $A$ (i.e. $x_i^a=\theta_i(a)x$ for all $a\in A$), the set 
$O:=\{\theta_1,\dots,\theta_d\}$ is a $G/A$-orbit in $\hat A$ (on which $G$ acts in the standard way). 
For a monomial $m=x_1^{\alpha_1}\dots x_d^{\alpha_d}\in {\mathbb{F}}[U]$ denote by $\Phi(m)$ the sequence over $\hat A$ containing 
$\theta_i$ with multiplicity $\alpha_i$ for $i=1,\dots,d$, and no other elements. 
Obviously $m$ is $A$-invariant if and only if $\Phi(m)$ is a zero-sum sequence. 

For every $k \le |O|$ we choose representatives $S_{k,1},...,S_{k,r_k}$ from each $G/A$-orbit of the $k$-element subsets of $O$. 
By Lemma \ref{lemma:tarto} we can assign to each of them an $A$-invariant monomial $m_{S_{k,i}}\in{\mathbb{F}}[U]$ with support  
${\mathrm{supp}}(\Phi(m_{S_{k,i}}))=S_{k,i}$ and  degree at most $p$. 
Now consider the polynomials:
\begin{align*} \label{hsop ZpZq}
f_k = \sum_{i=1}^{r_k} \sum_{g\in G/A}m_{S_{k,i}}^g  \qquad \text{ for } k =1,...,|O| \end{align*}
They are all $G$-invariants, 
moreover, it is easily checked that their common zero locus is $\{ 0 \}$. 
Indeed, if the vector $u=(u_1,...,u_{|O|}) \in \bar{{\mathbb{F}}}\otimes U\cong \bar{{\mathbb{F}}}^d$ belongs to this common zero locus,
and if the set $S = \{i: u_i \neq 0 \}$ has cardinality $k > 0$ then $ 0= f_k(u)  = c\cdot m_S(u)$ for some divisor $c$ of $d$. It follows that $m_S(u)=0$,  
implying that $u_j = 0$ for an index $j \in S$, which is a contradiction. 
Consequently ${\mathbb{F}}[U]^G$ is finitely generated over ${\mathbb{F}}[f_1,...,f_{|O|}]$ by Proposition~\ref{prop:hilbert zero}, 
hence $\sigma(G,U) \le \max_k \deg(f_k) \le p$.
\end{proof}

\begin{proposition}\label{sigma ZpZ2}
Let $G = A \rtimes Z_2$ be  a semidirect product where $A$ is a non-trivial abelian group on which $Z_2$ acts by inversion. 
Then $\sigma(G) = \exp(A)$. 
\end{proposition}

\begin{proof}
By Lemma~\ref{lemma:sigma irred} we know that $\sigma(G) = \max \sigma(G,U)$ where $U$ is an irreducible $G$-module. 
As above, we may assume that ${\mathbb{F}}$ is algebraically closed. 
If $U$ is $1$-dimensional, then $\sigma(G,U)\le \sigma(G/G')$, where $G'$ is the commutator subgroup of $G$. 
It is easy to see that $G/G'$ is an elementary abelian $2$-group, whence $\sigma(G/G')=\exp(G/G')=2\le\exp(A)$. 
If the irreducible $G$-module $U$ is not $1$-dimensional, then ${\mathbb{F}}[U]_1 = U^* = \langle x,y\rangle$ where $x^a =\theta(a) x$ for any $a \in A$ and some character 
$\theta \in \hat{A}$, 
and $y^a = \theta(a)^{-1} y $; moreover $x$ and $y$ are exchanged by the generator $b$ of $Z_2$. 
Let $e$ denote the order of $\theta$ in $\hat{A}$; evidently $e \le \exp{A}$. Now it is easily seen that 
 ${\mathbb{F}}[U]^G \supseteq {\mathbb{F}}[x^e+ y^e, xy]$, whence $\sigma(G,U) \le \max\{e,2\}$. 
Thus  we proved the inequality $\sigma(G)\le\exp(A)$. For the reverse inequality note that $\sigma(G)\ge\sigma(A)$ by the first inequality in Lemma~\ref{lemma:sigma red2}. 
\end{proof}


\section{An improved general  bound on $\sigma(G)$}

In this section we  give an improvement for non-cyclic $G$ of  the general inequality $\sigma(G)\le |G|$.

\begin{theorem}\label{thm:mainsigma}
If $G$ is non-cyclic and $q$ is the smallest prime divisor of $|G|$, then
\[ \sigma(G) \le \frac{|G|}q \] 
\end{theorem}

\begin{proof} 
If $G$ has a subquotient isomorphic to $Z_p \times Z_p$ for some prime $p$ 
then by Corollary~\ref{cor:sigma subquotient} and Corollary~\ref{cor: sigma abelian} we get that:
\[ \frac{\sigma(G)}{|G|} \le \frac{\sigma(Z_p \times Z_p )}{p^2} = \frac{1}{p} \]
and we are done. 
Note that for a non-cyclic $p$-group $P$, the factor group $P/\Phi(P)$ (where $\Phi(P)$ is the Frattini subgroup of $P$) contains a subgroup isomorphic 
to $Z_p\times Z_p$. So it remains to deal with the case when all Sylow subgroups of $G$ are cyclic. 
Then by a well known theorem of Burnside $G$ is the semidirect product of cyclic groups, and necessarily contains as a subquotient a 
non-abelian semi-direct product $Z_p\rtimes Z_q$, where $p,q$ are  primes, $q$ dividing $p-1$ (see for example \cite{cziszter-domokos:1} for references and details). 
Therefore by Corollary~\ref{cor:sigma subquotient} and  Proposition~\ref{sigma ZpZq}   we conclude  
$\sigma(G)/|G|\le \sigma(Z_p\rtimes Z_q)/pq=1/q$. 
\end{proof}

\begin{remark} Theorem~\ref{thm:mainsigma} is sharp for example for the abelian group $Z_{nq}\times Z_q$ or for a non-abelian semi-direct product 
$Z_n\rtimes Z_q$. 
\end{remark}  

Theorem~\ref{thm:mainsigma} is a variant for $\sigma(G)$ of the main combined result of \cite{cziszter-domokos:1} and \cite{cziszter-domokos:2} concerning $\beta(G)$.  
The present result for $\sigma(G)$ is  easier, but the conclusion is stronger. 
We finish by stating a  corresponding conjectured statement for the Noether number: 

\begin{conjecture} 
Let $\mathcal{C}_q$ denote the set of isomorphism classes of non-cyclic finite groups of order not divisible by  ${\mathrm{char}}({\mathbb{F}})$
and with smallest prime divisor  $q$. Then 
\[ \limsup_{G \in \mathcal{C}_q}  \frac{\beta(G)}{|G|} = \frac{1}{q} \]
\end{conjecture}

The case $q=2$ holds by \cite{cziszter-domokos:1} and \cite{cziszter-domokos:2}. For $q>2$ the conjecture is open. 

\begin{center} {\bf Acknowledgement}\end{center} 

We thank the referees for helpful comments on the manuscript. 



\begin{thebibliography}{mmm}

\bibitem{benson} D. J. Benson, Polynomial Invariants of Finite Groups, 
Cambridge University  Press, 1993. 


\bibitem{cziszter-domokos:1} K. Cziszter and M. Domokos, Groups with large  Noether bound, arXiv:1105.0679v4 .

\bibitem{cziszter-domokos:2} K. Cziszter and M. Domokos, Noether's bound for the groups with a cyclic subgroup of index two, arXiv:1205.3011v1. 

\bibitem{delorme}  Ch. Delorme, O. Ordaz, and D. Quiroz, Some remarks on Davenport constant, Discrete Math. 237 (2001), 119-128.


\bibitem{derksen} H. Derksen, Polynomial bounds for rings of invariants, Proc. Amer. Math. Soc. 129 (2001), 955-963. 

\bibitem{derksen-kemper} 
H. Derksen and G. Kemper, 
Computational Invariant Theory, 
Springer-Verlag, Berlin, 2002. 


\bibitem{domokos-hegedus} M. Domokos and  P. Heged\H us, 
Noether's bound for polynomial invariants of finite groups, 
Arch. Math. 74 (2000), 161-167. 


\bibitem{fleischmann} P. Fleischmann, The Noether bound in invariant theory of finite groups, Adv. Math. 156 (2000), 23-32. 

\bibitem{fleischmann:2} P. Fleischmann, On invariant theory of finite groups,  in "Invariant Theory in All Characteristics", (Ed.: H. E. A. Eddy Campbell and D. L. Wehlau), 
CRM Proceedings and Lecture Notes 35, Amer. Math. Soc., Providence, Rhode Island, pp. 43-69, 2004. 


\bibitem{fogarty} J. Fogarty, On Noether's bound for polynomial invariants of a finite group, Electron. Res. Announc. Amer. Math. Soc. 7 (2001), 5-7. 
 

\bibitem{freeze-schmid} M. Freeze and W. A. Schmid, Remarks on a generalization of the Davenport constant, 
Discrete Mathematics 310 (2010) 3373-3389. 

\bibitem{geroldinger-gao} W. Gao and A. Geroldinger, Zero-sum problems in finite abelian groups: A survey, 
Expo. Math. 24 (2006), 337-369. 



\bibitem{geroldinger-halterkoch}  A. Geroldinger and F. Halter-Koch, Non-unique factorizations. Algebraic, combinatorial and analytic theory, Monographs and Textbooks in Pure and Applied Mathematics, Chapman \& Hall/CRC, 2006.


\bibitem{halter-koch} F. Halter-Koch, 
A generalization of Davenport's constant and its arithmetical applications, 
Colloquium Mathematicum 63 (1992), 203-210. 

\bibitem{kemper} G. Kemper, Separating invariants,  J. Symbolic Computation 44 (2009), 1212-1222. 

\bibitem{knop} F. Knop, On Noether's and Weyl's bound in positive characteristic, 
in "Invariant Theory in All Characteristics", (Ed.: H. E. A. Eddy Campbell and D. L. Wehlau), 
CRM Proceedings and Lecture Notes 35, Amer. Math. Soc., Providence, Rhode Island, pp. 175-188, 2004. 

\bibitem{kohls-kraft} M. Kohls and H. Kraft, Degree bounds for separating invariants, Math. Res. Lett. 17 (2010), 1171-1182. 

\bibitem{noether:1916} E. Noether, Der Endlichkeitssatz der Invarianten endlicher Gruppen, 
Math. Ann. 77 (1916), 89-92. 


\bibitem{popov} V. Popov, The constructive theory of invariants, Math. USSR Izvest. 10 (1982), 359-376.

\bibitem{schmid} B. J. Schmid, Finite groups and invariant theory, 
``Topics in Invariant Theory'', Lect. Notes in Math. 1478 (1991), 
35-66. 

\bibitem{sezer} M. Sezer, Sharpening the generalized Noether bound in the invariant theory of finite groups, J. Algebra 254 (2002), 252-263. 

\bibitem{weyl} H. Weyl, The Classical Groups, Princeton University Press, Princeton, 1939. 





\end{thebibliography}
\end{document}